\documentclass[reqno,12pt]{amsart}
\usepackage{bbm}
\usepackage[all,pdf]{xy}
\usepackage{epsfig}
\usepackage{amsmath}
\usepackage{amssymb}
\usepackage{amscd}
\usepackage{graphicx}

\allowdisplaybreaks[4]

\makeatletter
\@namedef{subjclassname@2020}{%
	\textup{2020} Mathematics Subject Classification}
\makeatother
\usepackage[colorlinks=true]{hyperref}
\usepackage{amsfonts}
\usepackage[top=25mm, bottom=27mm, left=27mm, right=27mm]{geometry}

\newtheorem{thm}{Theorem}[section]
\newtheorem{lemma}[thm]{Lemma}
\newtheorem{ex}[thm]{Example}

\newtheorem{prop}[thm]{Proposition}

\newtheorem{ques}[thm]{Question}
\newtheorem{cor}[thm]{Corollary}

\newtheorem{rem}[thm]{Remark}

\newtheorem{con}{Conjecture}

\def \N {\mathbb N}

\def \Z {\mathbb Z}
\def \R {\mathbb R}

\def\B {\mathcal B}
\numberwithin{equation}{section}

\begin{document}

	\baselineskip 14pt
	
	\title[]{Polynomial ergodic averages of measure-preserving systems acted by $\Z^{d}$}
	
	\author{Rongzhong Xiao}
	
	\address{School of Mathematical Sciences, University of Science and Technology of China, Hefei, Anhui, 230026, PR China}
	\email{xiaorz@mail.ustc.edu.cn}


	 \begin{abstract}
	In this paper, we reduce pointwise convergence of polynomial ergodic averages of general measure-preserving system acted by $\Z^d$ to the case of  measure-preserving system acted by $\Z^d$ with zero entropy. As an application, we can build pointwise convergence of polynomial ergodic averages for $K$-system acted by $\Z^d$.
	\end{abstract}
		\maketitle
		
	\section{Introduction}
	Throughout the paper, let $(X,\B,\mu)$ be a Lebesgue space. Let $\Z[n]$ denote all polynomials with integer coefficients.

	Let $G$ be an infinte, countable, discrete group. A tuple $(X,\B,\mu,G)$ is a measure-preserving system if there exists a group homomorphism $\Pi:G\rightarrow \rm{MPT}(X,\B,\mu)$ where $\rm{MPT}(X,\B,\mu)$ denotes the group of invertible measure-preserving transformations of $(X,\B,\mu)$. We write $g$ for the measure-preserving transformation $\Pi(g)$. When $G=\Z$, we write $(X,\B,\mu,T)$ for measure-preserving system $(X,\B,\mu,G)$ where $T$ denotes  the measure-preserving transformation $\Pi(1)$. 
	
	In 1977, H. Furstenberg provided an ergodic theoretic proof for Szemer{\'e}di's theorem in \cite{F}. From then on, convergence of polynomial ergodic averages of measure-preserving system was a fundamental part of study of ergodic theory. Normally, we focus on the following question.
	\begin{ques}
		$($Furstenberg–Bergelson–Leibman conjecture \cite[Page 468]{BVLA}$)$Let $d,m\in \N$ be given. Let $T_1,\cdots,T_d:X\rightarrow X$ be a family of invertible measure-preserving transformations of $(X,\B,\mu)$ that generates a nilpotent group. Is it true that for any $p_{i,j}(n)\in \Z[n], 1\le i\le d, 1\le j\le m$ and for any $f_1,\cdots,f_m\in L^{\infty}(X,\B,\mu)$, $$\lim_{N\rightarrow \infty}\frac{1}{N}\sum_{n=0}^{N-1}\prod_{j=1}^{m}f_j(T_{1}^{p_{1,j}(n)}\cdots T_{d}^{p_{d,j}(n)}x)$$ exists in $L^{2}(\mu)$-norm or almost everywhere? 
	\end{ques}
	
	For $L^{2}(\mu)$-norm case, the question had been solved completely after some works. For $d=1$, H.Furstenberg and B. Weiss dealt with the form $\frac{1}{N}\sum_{n=0}^{N-1}f_1(T^{n^2}x)f_2(T^{n}x)$ in \cite{FW}. B. Host and B. Kra gave a answer for linear polynomials and described the structure of charateristic factors in \cite{HK}. Later, A. Leibman extend the result to general polynomials in \cite{Le}.
	For commute transformations, when $d=m=2$, J.-P. Conze and E. Lesigne constructed an answer on $L^{1}(\mu)$-norm for linear polynomials in \cite{CL}. T. Tao built the result for such form $\frac{1}{N}\sum_{n=0}^{N-1}\prod_{j=1}^{d}f_j(T_{j}^{n}x)$ in \cite{T}. T. Austin established same result for this form by pleasant extension in \cite{Au}.
	Finally, M. N. Walsh gave a complete answer for  measure-preserving system acted by nilpotent group in \cite{W}. 
	
	For pointwise case, there were some progressions for this conjecture over the last few decades. When $d=m=1$ or $d=1,m=2$ with $p_{1,1}(n)=an,p_{1,2}(n)=bn$ where $a$ and $b$ are non-zero integers and $a\neq b$, the question is solved by J. Bourgain in \cite{B2} and \cite{B4} respectively. For distal system acted by $\Z$, W. Huang, S. Shao and X. Ye answered it for linear polynomials in \cite{HSY}. For distal system acted by $\Z^d$ where $d>1$, S. Donoso and W. Sun established the similar result in \cite{DS}. Recently, when $d=1,m=2$ and $p_{1,1}(n)=n,\deg p_{1,2}(n)\ge 2$, B. Krause, M. Mirek and T. Tao constructed corresponding result in \cite{BMT}. To know more results, one can refer to \cite{A,B3,B1,IAMMS,L1,L2,N}.
	
	In 1996, J. Derrien and E. Lesigne reduced the pointwise convergence of polynomial ergodic averages of general measure-preserving system acted by $\Z$ to the case of  measure-preserving system acted by $\Z$ with zero entropy in \cite{DL}. Here, we state the result specifically. 
	 \begin{thm}
	 	Let $T:X\rightarrow X$ be an invertible measure-preserving transformations of $(X,\B,\mu)$. Given $m\in \N$, let $p_{j}(n)\in \Z[n], 1\le j\le m$. Let $P_{\mu}(T)$ be Pinsker $\sigma$-algebra$($for definition, see Subsection \ref{K}$)$ of measure-preserving system $(X,\B,\mu,T)$. Then for any $f_1,\cdots,f_m\in L^{\infty}(X,\B,\mu)$, $$\lim_{N\rightarrow \infty}\frac{1}{N}\sum_{n=0}^{N-1}\prod_{j=1}^{m}T^{p_{j}(n)}f_j$$ exists almost everywhere if and only if for any $h_1,\cdots,h_m\in L^{\infty}(X,P_{\mu}(T),\mu)$, $$\lim_{N\rightarrow \infty}\frac{1}{N}\sum_{n=0}^{N-1}\prod_{j=1}^{m}T^{p_{j}(n)}h_j$$ exists almost everywhere. 
	 \end{thm}
	In this paper, we extend the above result to measure-preserving system acted by $\Z^d$. That is,
	 \begin{thm}\label{T1}
	 	Let $d\in \N$ be given. Let $T_1,\cdots,T_d:X\rightarrow X$ be a family of invertible measure-preserving transformations of $(X,\B,\mu)$ that generates $\Z^d$. Given $m\in \N$, let $p_{i,j}(n)\in \Z[n], 1\le i\le d, 1\le j\le m$. Let $P_{\mu}(\Z^d)$ be Pinsker $\sigma$-algebra of measure-preserving system $(X,\B,\mu,\Z^d)$. Then for any $f_1,\cdots,f_m\in L^{\infty}(X,\B,\mu)$, $$\lim_{N\rightarrow \infty}\frac{1}{N}\sum_{n=0}^{N-1}\prod_{j=1}^{m}f_j(T_{1}^{p_{1,j}(n)}\cdots T_{d}^{p_{d,j}(n)}x)$$ exists almost everywhere if and only if for any $h_1,\cdots,h_m\in L^{\infty}(X,P_{\mu}(\Z^d),\mu)$, $$\lim_{N\rightarrow \infty}\frac{1}{N}\sum_{n=0}^{N-1}\prod_{j=1}^{m}h_j(T_{1}^{p_{1,j}(n)}\cdots T_{d}^{p_{d,j}(n)}x)$$ exists almost everywhere. 
	 \end{thm}
	 
	 By the idea of the arguments of the proof of Theorem \ref{T1}, we can get the following result.
	 \begin{thm}
	 	Let $d\in \N$ be given. Let $T_1,\cdots,T_d:X\rightarrow X$ be a family of invertible measure-preserving transformations of $(X,\B,\mu)$ that generates $\Z^d$. Given $m\in \N$, let $p_{i,j}(n)\in \Z[n], 1\le i\le d, 1\le j\le m$. Let $g:\N\rightarrow \R$ be a bounded sequence and $-\infty<\lim_{N\rightarrow \infty}\frac{1}{N}\sum_{n=1}^{N}g(n)=a<\infty$. Let $P_{\mu}(\Z^d)$ be Pinsker $\sigma$-algebra of measure-preserving system $(X,\B,\mu,\Z^d)$. Then for any $f_1,\cdots,f_m\in L^{\infty}(X,\B,\mu)$, $$\lim_{N\rightarrow \infty}\frac{1}{N}\sum_{n=1}^{N}g(n)\prod_{j=1}^{m}f_j(T_{1}^{p_{1,j}(n)}\cdots T_{d}^{p_{d,j}(n)}x)$$ exists almost everywhere if and only if for any $h_1,\cdots,h_m\in L^{\infty}(X,P_{\mu}(\Z^d),\mu)$, $$\lim_{N\rightarrow \infty}\frac{1}{N}\sum_{n=1}^{N}g(n)\prod_{j=1}^{m}h_j(T_{1}^{p_{1,j}(n)}\cdots T_{d}^{p_{d,j}(n)}x)$$ exists almost everywhere.
	 \end{thm}
	 
	 By the remark after proof of Theorem \ref{T1} and Theorem \ref{T1}, we can get the following result.
	 \begin{cor}
	 	Let $d\in \N$ be given. Let $T_1,\cdots,T_d:X\rightarrow X$ be a family of invertible measure-preserving transformations of $(X,\B,\mu)$ that generates $\Z^d$. Given $m\in \N$, let $p_{i,j}(n)\in \Z[n], 1\le i\le d, 1\le j\le m$. If $(X,\B,\mu,\Z^d)$ is a $K$-system$($for definition, see Subsection \ref{K}$)$, then for any $f_1,\cdots,f_m\in L^{\infty}(X,\B,\mu)$, $$\lim_{N\rightarrow \infty}\frac{1}{N}\sum_{n=0}^{N-1}\prod_{j=1}^{m}f_j(T_{1}^{p_{1,j}(n)}\cdots T_{d}^{p_{d,j}(n)}x)$$ exists almost everywhere. In particular, if for each $1\le j\le m$, $T_{1}^{p_{1,j}(n)}\cdots T_{d}^{p_{d,j}(n)}$ is not a constant and for any $1\le k,l\le m$ with $k\neq l$, $T_{1}^{p_{1,k}(n)-p_{1,l}(n)}\cdots T_{d}^{p_{d,k}(n)-p_{d,l}(n)}$ is not a constant, then the limit function must be $\prod_{j=1}^{m}\int f_jd\mu$.
	 \end{cor}

	In the above, the polynomial ergodic averages go along the positive integers. Next, we want to consider the polynomial ergodic averages that go along the prime numbers. That is, $$\frac{1}{N}\sum_{n=0}^{N-1}\prod_{j=1}^{m}f_j(T_{1}^{p_{1,j}(a_n)}\cdots T_{d}^{p_{d,j}(a_n)}x)$$ where $\mathbb{P}=\{a_0<a_1<\cdots<a_n<\cdots\}$ consists of all prime numbers. 
	
	Here, we build a result for a class of special cases.
	
	\begin{thm}\label{T2}
		Let $d\in \N$ be given. Let $T_1,\cdots,T_d:X\rightarrow X$ be a family of invertible measure-preserving transformations of $(X,\B,\mu)$ that generates $\Z^d$. Given $m\in \N$, let $p_j(n)\in \Z[n],1\le j\le m$. Let $g:\{1,\cdots,m\}\rightarrow \{1,\cdots,d\}$. Let $P_{\mu}(\Z^d)$ be Pinsker $\sigma$-algebra of measure-preserving system $(X,\B,\mu,\Z^d)$. Then for any $f_1,\cdots,f_m\in L^{\infty}(X,\B,\mu)$, $$\lim_{N\rightarrow \infty}\frac{1}{N}\sum_{n=0}^{N-1}\prod_{j=1}^{m}f_j(T_{g(j)}^{p_{j}(a_n)}x)$$ exists almost everywhere if and only if for any $h_1,\cdots,h_m\in L^{\infty}(X,P_{\mu}(\Z^d),\mu)$, $$\lim_{N\rightarrow \infty}\frac{1}{N}\sum_{n=0}^{N-1}\prod_{j=1}^{m}h_j(T_{g(j)}^{p_{j}(a_n)}x)$$ exists almost everywhere where $\mathbb{P}=\{a_0<a_1<\cdots<a_n<\cdots\}$ consists of all prime numbers.
	\end{thm}
	
	As for general case, we expect the following result.
	\begin{con}\label{con2}
		Under assumption of Theorem \ref{T1}, then for any $f_1,\cdots,f_m\in L^{\infty}(X,\B,\mu)$, $$\lim_{N\rightarrow \infty}\frac{1}{N}\sum_{n=0}^{N-1}\prod_{j=1}^{m}f_j(T_{1}^{p_{1,j}(a_n)}\cdots T_{d}^{p_{d,j}(a_n)}x)$$ exists almost everywhere if and only if for any $h_1,\cdots,h_m\in L^{\infty}(X,P_{\mu}(\Z^d),\mu)$, $$\lim_{N\rightarrow \infty}\frac{1}{N}\sum_{n=0}^{N-1}\prod_{j=1}^{m}h_j(T_{1}^{p_{1,j}(a_n)}\cdots T_{d}^{p_{d,j}(a_n)}x)$$ exists almost everywhere where $\mathbb{P}=\{a_0<a_1<\cdots<a_n<\cdots\}$ consists of all prime numbers.
	\end{con}
	Based on the proof of Theorem \ref{T2} and Theorem \ref{T1}, if one can build maximal inequality like the form mentioned in the Theorem \ref{thm5} for prime sequence, then there exists an affirmative answer for the Conjecture \ref{con2}.
	
	The proofs of Theorem \ref{T1} and Theorem \ref{T2} can be divided into two steps. At first, we reduce the general essentially bounded measurable functions to indicator functions by maximal inequality$($Theorem \ref{thm5} and Theorem \ref{thm1}$)$. Secondly, we use the structure of Pinsker $\sigma$-algebra$($Proposition \ref{prop1}$)$ and probability theoretic method to reach our conclusion. The key points are how to choose a proper algebraic past of $\Z^d$ and using the algebraic past to describe structure of Pinsker $\sigma$-algebra.
	 
	\subsection*{Organization of the paper}Section 2 contains some prepared notions and results. In section 3, we prove Theorem \ref{T1}. In section 4, we show Theorem \ref{T2}.

	\section{Preliminaries}
In this section, we review some notions and some results that will be used in later proof. 
	
	\subsection{Conditional expectation}
	Let $\mathcal{A}$ be a sub-$\sigma$-algebra of $\B$. For any $f\in L^{1}(X,\B,\mu)$, there exists a function $\mathbb{E}(f|\mathcal{A})$ that satisfies the following properties: (1)$\mathbb{E}(f|\mathcal{A})$ is $\mathcal{A}$-measurable; (2)for any $A\in \mathcal{A}$, $\int _{A}fd\mu=\int _{A}\mathbb{E}(f|\mathcal{A})d\mu$. Clearly, $\mathbb{E}(f|\mathcal{A})$ is characterized almost everywhere. The map $$\mathbb{E}(\cdot|\mathcal{A}):L^{1}(X,\B,\mu)\rightarrow L^{1}(X,\mathcal{A},\mu)$$ is called conditional expectation. In particular, when $\mathcal{A}$ is trivival, we write $\mathbb{E}f$ for $\mathbb{E}(f|\mathcal{A})$ for any $f\in L^{1}(X,\B,\mu)$.
	
	\begin{thm}\label{thm4}
		$($\cite[Theorem 5.8]{EW}$)$If $\{\mathcal{A}_n\}_{n\ge1}$ is a sequence of sub-$\sigma$-algebra of $\B$ with $\mathcal{A}_1\subset\mathcal{A}_2\subset\cdots\subset\mathcal{A}_n\subset\cdots$ and $\bigcap_{n\ge 1}\mathcal{A}_n=\mathcal{A}$, then for any $f\in L^{1}(X,\B,\mu)$, $$\mathbb{E}(f|\mathcal{A}_n)\rightarrow \mathbb{E}(f|\mathcal{A}) $$ almost everywhere and in $L^{1}(\mu)$ as $n\rightarrow\infty$.
	\end{thm}
	
	Next, we give a simple lemma that will be used in the later proof.
	\begin{lemma}\label{lem}
		Let $\mathcal{A}$ be a sub-$\sigma$-algebra of $\B$. Let $G$ be an infinte, countable, discrete group. Let $(X,\B,\mu,G)$ be a measure-preserving system. Then for any $f\in L^{\infty}(X,\B,\mu)$ and $g\in G$, we have $$g\mathbb{E}(f|\mathcal{A})=\mathbb{E}(gf|g^{-1}\mathcal{A}).$$
	\end{lemma}

	\subsection{Entropy}
	
     Let $G$ be an infinite, countable, discrete amenable group. A sequence $\left\{F_{n}\right\}_{n=1}^{\infty}$ of non-empty finite subsets of $G$ is called a \textbf{F\o lner sequence} if for every $g \in G$,
	$
	\lim _{n \rightarrow+\infty} \frac{\left|g F_{n} \Delta F_{n}\right|}{\left|F_{n}\right|}=0,
	$
	where $|\cdot|$ denotes the cardinality of a set.
	
	A \textbf{measurable partition} of $(X,\B,\mu)$ is a disjoint collection of elements of $\B$ whose union is $X$.  Define $$\mathcal{P}^{\mu}_{X}=\{\alpha: \alpha\ is\ a\ finite\  measurable\  partition\ of\ (X,\B,\mu)\}.$$
	
	Given$\{\alpha_i:i\in \mathcal{I}\}\subset \mathcal{P}^{\mu}_{X}$ where $\mathcal{I}$ is an index set,  $\bigvee_{i\in \mathcal{I}}\alpha_i$ is a measurable partition with such a property that the sub-$\sigma$-algebra generated by $\bigvee_{i\in \mathcal{I}}\alpha_i$ is the minimal sub-$\sigma$-algebra including $\alpha_i,i\in \mathcal{I}$. When $\mathcal{I}$ is finite, we know  $\bigvee_{i\in \mathcal{I}}\alpha_i=\{\bigcap_{i\in \mathcal{I}}A_i:A_i\in \alpha_i\}$.
	
	Given $\alpha\in \mathcal{P}^{\mu}_{X}$, we define
	$
	H_{\mu}(\alpha)=\sum_{A \in \alpha}-\mu(A) \log \mu(A).
	$ Let $\mathcal{A}$ be a sub-$\sigma$-algebra of $\B$, we define $H_{\mu}(\alpha|\mathcal{A})=-\int\sum_{A \in \alpha}\mathbb{E}(1_{A}|\mathcal{A})\log\mathbb{E}(1_{A}|\mathcal{A})d\mu$.
	
	Let $(X,\B,\mu,G)$ be a measure-preserving system. The \textbf{measure-theoretic entropy of $\mu$ relative to $\alpha$} is defined by
	$$
	h_{\mu}(G, \alpha)=\lim _{n \rightarrow+\infty} \frac{1}{\left|F_{n}\right|} H_{\mu}\left(\bigvee_{g \in F_{n}} g^{-1} \alpha\right)
	$$
	where $\{F_{n}\}_{n=1}^\infty$ is a Fl\o lner sequence of the amenable group $G$.  By Theorem 6.1 of  \cite{LW}, the limit exists and is independent of choices of F\o lner sequences. The \textbf{measure-theoretic entropy of $\mu$} is defined by
	$$
	h_{\mu}(G)=h_{\mu}(G, X)=\sup _{\alpha \in \mathcal{P}_{X}^{\mu}} h_{\mu}(G, \alpha)
	.$$
	
	\subsection{Algebraic past}
	Given group $G$ with identity element $e_G$, an \textbf{algebraic past} of $G$ is a subset $\Phi$ with properties (1)-(3): (1) $\Phi \cap {\Phi}^{-1}$ is empty; (2) $\Phi \cup {\Phi}^{-1} \cup \{e_G\}=G$; (3) $\Phi \cdot \Phi \subset \Phi$. The group $G$ is \textbf{left-orderable} if there exists a linear ordering in $G$ which is invariant under left translation. The group $G$ is left-orderable if and only if there exists an algebraic past $\Phi$ in $G$. Indeed, one can obtain the desired linear-order based on $\Phi$ as follows: $g_1$ is less than $g_2$ (write $g_1 {<}_{\Phi} g_2$ or $g_1 <g_2$) if ${g_2}^{-1}g_1\in \Phi$. 
	\begin{ex}
		When $G=\Z$, let $\Phi$ be $\{n\in \Z:n<0\}$. Then we can get the natural linear order of $\Z$. 
	\end{ex}
	\begin{ex}\label{ex1}
		For each $d\in\N$ and for any $A_1,\cdots,A_d>0$, the $\Phi=\{(n_1,\cdots,n_d)\in\Z^d:there\ exists\ j\in\{0,1,\cdots,d-1\}such\ that\ \sum_{l=1}^{d-k}A_{l}n_l=0\ for\ k=0,\cdots,j-1\ and\ \sum_{l=1}^{d-j}A_{l}n_l<0\}$ is an algebraic past of $\Z^d$.
	\end{ex}
	
	\subsection{Pinsker $\sigma$-algebra}\label{K}
      Let $G$ be an infinite, countable, discrete amenable group. The \textbf{Pinsker $\sigma$-algebra} of measure-preserving system $(X,\B,\mu,G)$ is defined as $$P_{\mu}(G)=\{A\in \B:h_{\mu}(G,\{A,A^c\})=0\}.$$
	If $P_{\mu}(G)$ is trivival, we say that $(X,\B,\mu,G)$ is a \textbf{Kolmogorov system}(write $K$-system).
	
	When $G=\Z$, we can describe the structure of $P_{\mu}(G)$ based on the natural linear-order of $\Z$. Similarly, for genaral case, we can give the structure of $P_{\mu}(G)$ by algebraic past that plays same role as the natural linear-order of $\Z$.
	\begin{thm}\label{thm2}
		$($\cite[Theorem 3.1]{HXY}$)$Let $G$ be a countable discrete infinite amenable group with algebraic past $\Phi$, and let $(X,\B,\mu,G)$ be a measure-preserving system. Then for any $\alpha,\beta\in \mathcal{P}^{\mu}_{X}$,  $$h_{\mu}(G,\alpha\vee\beta)=h_{\mu}(G,\beta)+H_{\mu}(\alpha|\beta_{G}\vee\alpha_{\Phi}),$$ where $\beta_{G}=\bigvee_{g\in G}g\beta$ and $\alpha_{\Phi}=\bigvee_{g\in\Phi}g\alpha$.
	\end{thm}
	
	Based on the Theorem \ref{thm2}, we can get the following results.
	 \begin{prop}\label{prop1}
	 	Let $G$ be a countable discrete infinite amenable group with algebraic past $\Phi$  and identity element $e_{G}$, and let $(X,\B,\mu,G)$ be a measure-preserving system. Then $$P_{\mu}(G)\supset \bigvee_{\beta\in \mathcal{P}^{\mu}_{X}}\bigwedge_{g\in \Phi\cup\{e_G\}}g\beta_{\Phi}.$$
	 \end{prop}
	 \begin{proof}
	 	Select $\alpha\in \mathcal{P}^{\mu}_{X}$ and fix it. Choose $\beta\in \mathcal{P}^{\mu}_{X}$ such that $\beta \subset \bigwedge_{g\in \Phi\cup\{e_G\}}g\alpha_{\Phi}$. Then $\beta_{G}\subset \alpha_{\Phi}$. And $$H_{\mu}(\alpha\vee\beta|\alpha_{\Phi}\vee\beta_{\Phi})\le H_{\mu}(\alpha\vee\beta|\alpha_{\Phi})=H_{\mu}(\alpha|\alpha_{\Phi}).$$
	 	By Theorem \ref{thm2}, we know $h_{\mu}(G,\alpha)=H_{\mu}(\alpha|\alpha_{\Phi})$. Then  $$H_{\mu}(\alpha\vee\beta|\alpha_{\Phi}\vee\beta_{\Phi})=h_{\mu}(G,\beta)+H_{\mu}(\alpha|\beta_{G}\vee\alpha_{\Phi})=H_{\mu}(\alpha|\alpha_{\Phi})+H_{\mu}(\beta|\beta_{\Phi}).$$ So $\beta\subset P_{\mu}(G)$.
	 \end{proof}
	
	\subsection{Maximal ergodic theorem}
		\begin{thm}\label{thm5}
			$($\cite[Theorem 1.2.(iii)]{IAMMS}$)$Let $d_1\in \N$ be given. Let $T_1,\cdots,T_{d_1}:X\rightarrow X$ be a family of invertible measure-preserving transformations of $(X,\B,\mu)$ that generates a nilpoent group of step two. Assume that $p_1,\cdots,p_{d_1}\in \Z[n]$ and let $d_2=\max\{\deg p_j(n):1\le j\le d_1\}$. Then for any $f\in L^{p}(X,\B,\mu),p>1$, there exists a constant $c(d_1,d_2,p)$ such that $$\big|\big|\sup_{N\ge 1}|\frac{1}{N}\sum\limits_{n=0}^{N-1}f(T_{1}^{p_1(n)}\cdots T_{d_1}^{p_{d_1}(n)}x)|\big|\big|_{p}\le c(d_1,d_2,p)||f||_{p}.$$
		\end{thm}
		Based on the ideas of  proof of Corollary 2.2 of \cite{DL} and Theorem \ref{thm5}, we can get the following result.
		\begin{thm}\label{thm6}
			Let $d,m\in \N$ be given.  Let $T_1,\cdots,T_d:X\rightarrow X$ be a family of invertible measure-preserving transformations of $(X,\B,\mu)$ that generates a nilpoent group of step two. Assume that $p_{1,j},\cdots,p_{d,j}\in \Z[n],1\le j\le m$ and  let $q_j,1\le j\le m$ be positive real number with $\sum\limits_{j=1}^{m}\frac{1}{q_j}<1$. Then $$\{(f_1,\cdots,f_m): \lim\limits_{N\rightarrow\infty}\frac{1}{N}\sum_{n=0}^{N-1}\prod_{j=1}^{m}f_j(T_{1}^{p_{1,j}(n)}\cdots T_{d}^{p_{d,j}(n)}x)\ exists\ almost\ everywhere\}$$ is closed in $L^{q_1}(X,\B,\mu)\times\cdots\times L^{q_m}(X,\B,\mu)$.
		\end{thm}
	The following theorem focus on the polynomiall ergodic averages that go along primes.
	\begin{thm}\label{thm1}
	$($\cite[Theorem 1.2]{N}$)$Let $(X,\B,\mu,T)$ be a measure-preserving system. Given $q(n)\in\Z[n]$ and $p>1$, then for any $f\in L^{p}(X,\B,\mu)$, there exixts a constant $c(p,q,T)$ such that $$\big|\big|\sup_{N\ge 1}|\frac{1}{N}\sum\limits_{n=0}^{N-1}T^{q(a_n)}f|\big|\big|_{p}\le c(p,q,T)||f||_{p}$$ where $\mathbb{P}=\{a_0<a_1<\cdots<a_n<\cdots\}$ consists of all prime numbers.
	\end{thm}
       Based on the ideas of  proof of Corollary 2.2 of \cite{DL} and Theorem \ref{thm1}, we can get the following result.
		\begin{thm}\label{thm3}
			Let $G$ be an infinite, countable, discrete group. Let $(X,\B,\mu,G)$ be a measure-preserving system. Given $d\in \N$, let $q_j$ be positive real number with $\sum\limits_{j=1}^{d}\frac{1}{q_j}<1$, $p_j(n)\in\Z[n]$, and $g_j\in G,1\le j\le d$. Then $$\{(f_1,\cdots,f_d): \lim\limits_{N\rightarrow\infty}\frac{1}{N}\sum_{n=0}^{N-1}\prod_{j=1}^{d}g_{j}^{p_j(a_n)}f_j\ exists\ almost\ everywhere\}$$ is closed in $L^{q_1}(X,\B,\mu)\times\cdots\times L^{q_d}(X,\B,\mu)$ where $\mathbb{P}=\{a_0<a_1<\cdots<a_n<\cdots\}$ consists of all prime numbers.
		\end{thm}

	\section{The proof of Theorem \ref{T1}}
	First, we introduce a lemma that will be used in the proof.
	\begin{lemma}\label{lem1}
		$($\cite[Theorem 5.1.2]{C}$)$Let $X_n(n\ge 1)$ be a sequence of random variables of probability space $(\Omega,\mathcal{D},\nu)$ with $\mathbb{E}X_n=0$ and $\mathbb{E}|X_{n}|^2\le M$ where $M>0$. And for any $n,m\in \N$ with $m\neq n$, $\mathbb{E}X_nX_m=0$. Then $\frac{1}{n}S_n\rightarrow 0$ almost everywhere as $n\rightarrow \infty$ where $S_n=X_1+\cdots+X_n$.
	\end{lemma}

	Next, we will provide a simple argument for a special case to express our basic idea of the proof. 
	\begin{prop}
			Under the assumption of Theorem \ref{T1}, let $A$ and $B$ be two measurable sets with positive measure. Then there exists a sub-$\sigma$-algebra $\mathcal{A}$ of $P_{\mu}(\Z^d)$ $$\frac{1}{N}\sum_{n=0}^{N-1}T_{1}^{3n^2} T_{2}^{8n^2}1_BT_{1}^{n^2} T_{2}^{-n^2}1_{A}-\frac{1}{N}\sum_{n=0}^{N-1}T_{1}^{3n^2} T_{2}^{8n^2}\mathbb{E}(1_B|\mathcal{A})T_{1}^{n^2} T_{2}^{-n^2}\mathbb{E}(1_A|\mathcal{A})\rightarrow 0$$ almost everywhere as $N\rightarrow \infty$.
	\end{prop}
	\begin{proof}
		Clearly, we know $\deg(n^2+2(-n^2))\ge 1$, $\deg(3n^2+2(8n^2))\ge 1$ and $\deg((3n^2-n^2)+2(8n^2-(-n^2)))\ge 1$. By Example \ref{ex1}, we know $\Phi=\{(n_1,\cdots,n_d)\in\Z^d:there\ exists\ j\in\{0,1,\cdots,d-1\}such\ that\ \sum_{l=1}^{d-k}A_{l}n_l=0\ for\ k=0,\cdots,j-1\ and\ \sum_{l=1}^{d-j}A_{l}n_l<0\}$ is an algebraic past of $\Z^d$ when $A_i=1$ for any $1\le i\le d$ with $i\neq 2$ and $A_2=2$. 
		
		To be convenient, we view $\vec{n}$ as $T_{1}^{n_1}\cdots T_{d}^{n_d}$ where $\vec{n}=(n_1,\cdots,n_d)\in \Z^d$. When $n>0$, we have $$(n^2,-n^2,0,\cdots,0)<_{\Phi}(3n^2,8n^2,0,\cdots,0).$$
		
		For any $g\in \Z^d$, $\mathcal{A}_{g}$ is a sub-$\sigma$-algebra generated by $\{h1_{A}:h\in \Z^d\ and\ g\le_{\Phi} h\}$ and $\{h1_{B}:h\in \Z^d\ and\ g\le_{\Phi} h\}$. Let $\mathcal{A}=\bigcap_{g\in \Z^d}\mathcal{A}_{g}.$ Note that $\Phi^{-1}$ is still an algebraic past of $\Z^{d}$. Use the definition of $\mathcal{A}$ and Proposition \ref{prop1}, we know that $\mathcal{A}\subset P_{\mu}(\Z^d).$
		
		Next, we verify that $$\lim\limits_{N\rightarrow \infty}\frac{1}{N}\sum_{n=0}^{N-1}1_B(T_{1}^{3n^2} T_{2}^{8n^2}x)(1_{A}-\mathbb{E}(1_A|\mathcal{A}_{g}))(T_{1}^{n^2} T_{2}^{-n^2}x)=0$$ almost everywhere for any $g\in \Phi^{-1}$.
		
		Let $\mathcal{J}_{n}=\mathcal{A}_{g+(n^2,-n^2,0,\cdots,0)}$. Let $X_n=1_B(T_{1}^{3n^2} T_{2}^{8n^2}x)(1_{A}-\mathbb{E}(1_A|\mathcal{A}_{g}))(T_{1}^{n^2} T_{2}^{-n^2}x)$. Then there exists $N\in \N$ such that when $n>N$, $(1_{A}-\mathbb{E}(1_A|\mathcal{A}_{g}))(T_{1}^{n^2} T_{2}^{-n^2}x)$ is $\mathcal{J}_{n+1}$-measurable and $1_B(T_{1}^{3n^2} T_{2}^{8n^2}x)$ is $\mathcal{J}_{n}$-measurable. 
		
		That is, when $n>N$, $X_n$ is $\mathcal{J}_{n+1}$-measurable. And $\mathbb{E}(X_n|\mathcal{J}_{n})=0$. 
		
		For any $n>m>N$, we have $$\mathbb{E}(X_nX_m)=\mathbb{E}(X_m\mathbb{E}(X_n|\mathcal{J}_{n}))=0.$$
		
		By Lemma \ref{lem1}, we get $$\lim\limits_{N\rightarrow \infty}\frac{1}{N}\sum_{n=0}^{N-1}1_B(T_{1}^{3n^2} T_{2}^{8n^2}x)(1_{A}-\mathbb{E}(1_A|\mathcal{A}_{g}))(T_{1}^{n^2} T_{2}^{-n^2}x)=0$$ almost everywhere.
		
		Likely, we have $$\lim\limits_{N\rightarrow \infty}\frac{1}{N}\sum_{n=0}^{N-1}(1_{B}-\mathbb{E}(1_B|\mathcal{A}_{g}))(T_{1}^{3n^2} T_{2}^{8n^2}x)\mathbb{E}(1_A|\mathcal{A})(T_{1}^{n^2} T_{2}^{-n^2}x)=0$$ almost everywhere for any $g\in \Phi^{-1}$.
		By Theorem \ref{thm5}, for any $g\in \Phi^{-1}$, we have 
		\begin{equation*}
		\begin{split}
		&\int \Big(\limsup_{N\rightarrow \infty}\Big|\frac{1}{N}\sum_{n=0}^{N-1}1_B(T_{1}^{3n^2} T_{2}^{8n^2}x)(1_{A}-\mathbb{E}(1_A|\mathcal{A}))(T_{1}^{n^2} T_{2}^{-n^2}x)\Big|\Big)d\mu\\ &\le \int \Big(\limsup_{N\rightarrow \infty}\Big|\frac{1}{N}\sum_{n=0}^{N-1}1_B(T_{1}^{3n^2} T_{2}^{8n^2}x)(1_{A}-\mathbb{E}(1_A|\mathcal{A}_{g}))(T_{1}^{n^2} T_{2}^{-n^2}x)\Big|\Big)d\mu\\&+\int \Big(\limsup_{N\rightarrow \infty}\Big|\frac{1}{N}\sum_{n=0}^{N-1}1_B(T_{1}^{3n^2} T_{2}^{8n^2}x)(\mathbb{E}(1_A|\mathcal{A}_{g})-\mathbb{E}(1_A|\mathcal{A}))(T_{1}^{n^2} T_{2}^{-n^2}x)\Big|\Big)d\mu\\ &\le c(d,2,2)||\mathbb{E}(1_A|\mathcal{A}_{g})-\mathbb{E}(1_A|\mathcal{A})||_{2}.
		\end{split}
		\end{equation*}
		
		By Theorem \ref{thm4}, we know that for $\mu$-a.e. $x\in X$,
		\begin{equation}\label{eq7}
		\lim_{N\rightarrow \infty}\frac{1}{N}\sum_{n=0}^{N-1}1_B(T_{1}^{3n^2} T_{2}^{8n^2}x)(1_{A}-\mathbb{E}(1_A|\mathcal{A}))(T_{1}^{n^2} T_{2}^{-n^2}x)=0.
		\end{equation}
		
		Likely, we can know that for $\mu$-a.e. $x\in X$,
		\begin{equation}\label{eq8}
		\lim\limits_{N\rightarrow \infty}\frac{1}{N}\sum_{n=0}^{N-1}(1_{B}-\mathbb{E}(1_B|\mathcal{A}))(T_{1}^{3n^2} T_{2}^{8n^2}x)\mathbb{E}(1_A|\mathcal{A})(T_{1}^{n^2} T_{2}^{-n^2}x)=0.
		\end{equation}
		
		Take sum for (\ref{eq7}) and (\ref{eq8}) and we have $$\frac{1}{N}\sum_{n=0}^{N-1}T_{1}^{3n^2} T_{2}^{8n^2}1_BT_{1}^{n^2} T_{2}^{-n^2}1_{A}-\frac{1}{N}\sum_{n=0}^{N-1}T_{1}^{3n^2} T_{2}^{8n^2}\mathbb{E}(1_B|\mathcal{A})T_{1}^{n^2} T_{2}^{-n^2}\mathbb{E}(1_A|\mathcal{A})\rightarrow 0$$ almost everywhere as $N\rightarrow \infty$. This finishes the proof.
	\end{proof}

		Clearly, the special case reflects a fact that we can use the polynomial ergodic averages that come from $(X,P_{\mu}(\Z^d),\mu)$ to consider the polynomial ergodic averages come from $(X,\B,\mu)$.

	Now, let us be back to the proof of Theorem \ref{T1}. Before this, we need a lemma.
	In the following lemma, we consider such situation: 
	
 Under the assumption of Theorem \ref{T1}. Let $f_1,\cdots,f_m\in L^{\infty}(X,\B,\mu)$ and fix them. Assume that $p_{i,j}(0)=0,1\le i\le d,1\le j\le m$. To be convenient, we view $\vec{n}$ as $T_{1}^{n_1}\cdots T_{d}^{n_d}$ where $\vec{n}=(n_1,\cdots,n_d)\in \Z^d$. Moreover, we assume that for each $1\le j\le m$, $(p_{1,j}(n),\cdots,p_{d,j}(n))$ is not a constant and for any $1\le k,l\le m$ with $k\neq l$, $(p_{1,k}(n)-p_{1,l}(n),\cdots,p_{d,k}(n)-p_{1,l}(n))$ is not a constant. Clearly, there exist $A_1,\cdots,A_d>0$ such that there exists $N_0\in \N$ such that when $n>N_0$, for each $1\le j\le m$, $$\deg\big(\sum_{i=1}^{d}A_i p_{i,j}(n)\big)\ge 1$$ and there exists $N_1\in \N$ such that when $n>N_1$, for any $1\le k,l\le m$ with $k\neq l$, $$\deg\big(\sum_{i=1}^{d}A_i(p_{i,k}(n)-p_{i,l}(n))\big)\ge 1.$$ 
	
	By Example \ref{ex1}, we know $\Phi=\{(n_1,\cdots,n_d)\in\Z^d:there\ exists\ j\in\{0,1,\cdots,d-1\}such\ that\ \sum_{l=1}^{d-k}A_{l}n_l=0\ for\ k=0,\cdots,j-1\ and\ \sum_{l=1}^{d-j}A_{l}n_l<0\}$ is an algebraic past of $\Z^d$. To be convenient, suppose that $$(p_{1,l}(n),\cdots,p_{d,l}(n)) <_{\Phi}(p_{1,k}(n),\cdots,p_{d,k}(n))$$ for any $1\le k,l\le m$ with $k<l$ when $n>N_2$ where $N_2\in \N$. For any $g\in \Z^d$, $\mathcal{A}_{g}$ is a sub-$\sigma$-algebra generated by $\{hf_l:1\le l\le d,\ h\in \Z^d\ and\ g\le_{\Phi} h\}$. Let $\mathcal{A}=\bigcap_{g\in \Z^d}\mathcal{A}_{g}.$
	 \begin{lemma}\label{prop3}
	 	For any $ j\in\{1,\cdots,m\}$, we have
	 	\begin{equation}\label{eq2}
	 	\begin{split}
	 	\frac{1}{N}\sum_{n=0}^{N-1}&\prod_{k=1}^{j-1}f_k(T_{1}^{p_{1,k}(n)}\cdots T_{d}^{p_{d,k}(n)}x)(f_j-\mathbb{E}(f_j|\mathcal{A}))(T_{1}^{p_{1,j}(n)}\cdots T_{d}^{p_{d,j}(n)}x)\cdot\\&\ \ \ \ \ \ \ \ \prod_{l=j+1}^{m}\mathbb{E}(f_{l}|\mathcal{A})(T_{1}^{p_{1,l}(n)}\cdots T_{d}^{p_{d,l}(n)}x)\rightarrow 0
	 	\end{split}
	 	\end{equation}
	 	almost everywhere as $N\rightarrow \infty$.
	  \end{lemma}
	\begin{proof}
		For any $\epsilon>0$, there exists $g_{j_0}\in \Phi^{-1}$ such that $||\mathbb{E}(f_j|\mathcal{A}_{g_{j_0}})-\mathbb{E}(f_j|\mathcal{A})||_{2}<\epsilon$ by Theorem \ref{thm4}. By Theorem \ref{thm5}, let $\tilde{d}=\max\{\deg p_{i,j}(n):1\le i\le d\}$, then we have
		\begin{equation*}
		\begin{split}
		& \int \Big(\limsup\limits_{N\rightarrow \infty}\Big|\frac{1}{N}\sum_{n=0}^{N-1}\prod_{k=1}^{j-1}f_k(T_{1}^{p_{1,k}(n)}\cdots T_{d}^{p_{d,k}(n)}x) (\mathbb{E}(f_j|\mathcal{A}_{g_{j_0}})-\mathbb{E}(f_j|\mathcal{A}))\\&(T_{1}^{p_{1,j}(n)}\cdots T_{d}^{p_{d,j}(n)}x)
		\prod_{l=j+1}^{m}\mathbb{E}(f_{l}|\mathcal{A}))(T_{1}^{p_{1,l}(n)}\cdots T_{d}^{p_{d,l}(n)}x)\Big|\Big)d\mu\\
		&\le \prod_{k=1}^{j-1}||f_k||_{\infty}\Big|\Big|\sup_{N\ge 1}\frac{1}{N}\sum_{n=0}^{N-1}|\mathbb{E}(f_j|\mathcal{A}_{g_{j_0}})-\mathbb{E}(f_j|\mathcal{A})|(T_{1}^{p_{1,j}(n)}\cdots T_{d}^{p_{d,j}(n)}x)\Big|\Big|_{2}\cdot\\&\prod_{l=j+1}^{m}||f_{l}||_{\infty}
		\le \prod_{k=1}^{j-1}||f_k||_{\infty} c(d,\tilde{d},2)\epsilon\prod_{l=j+1}^{m}||f_{l}||_{\infty}.
		\end{split}
		\end{equation*}
		
		So we only need to verify that as $N\rightarrow \infty$, for $\mu$-a.e. $x\in X$,
		\begin{equation}\label{eq3}
		\begin{split}
		\frac{1}{N}\sum_{n=0}^{N-1}&\prod_{k=1}^{j-1}f_k(T_{1}^{p_{1,k}(n)}\cdots T_{d}^{p_{d,k}(n)}x)(f_j-\mathbb{E}(f_j|\mathcal{A}_{g_{j_0}}))(T_{1}^{p_{1,j}(n)}\cdots T_{d}^{p_{d,j}(n)}x)\cdot\\&\ \ \ \ \ \ \ \ \prod_{l=j+1}^{m}\mathbb{E}(f_{l}|\mathcal{A})(T_{1}^{p_{1,l}(n)}\cdots T_{d}^{p_{d,l}(n)}x)\rightarrow 0.
		\end{split}
		\end{equation}
		
		Let $K=\max\{N_2,N_1,N_0\}$. There exists $L_j\in \N$ with $L_j>K$ and $Q\in \N$ such that when $n>L_j$, we have one of the following two equalites: $$g_{j_0}+(p_{1,j}((n-1)Q+i),\cdots,p_{d,j}((n-1)Q+i))<_{\Phi}(p_{1,j}(nQ+i),\cdots,p_{d,j}(nQ+i)),$$ 
 $$g_{j_0}+(p_{1,j}((n+1)Q+i),\cdots,p_{d,j}((n+1)Q+i))<_{\Phi}(p_{1,j}(nQ+i),\cdots,p_{d,j}(nQ+i))$$ for any $0\le i\le Q-1$ since when $n>N_0$, $\deg\big(\sum_{i=1}^{d}A_i p_{i,j}(n)\big)\ge 1$. Choose any $0\le i\le Q-1$ and fix it.
		Let 
		\begin{equation*}
		\begin{split}
		&\ \ \ \ \ \ \ \ X_{j,n}=\prod_{k=1}^{j-1}f_k(T_{1}^{p_{1,k}(nQ+i)}\cdots T_{d}^{p_{d,k}(nQ+i)}x)(f_j-\mathbb{E}(f_j|\mathcal{A}_{g_{j_0}}))\\&(T_{1}^{p_{1,j}(nQ+i)}\cdots T_{d}^{p_{d,j}(nQ+i)}x) \prod_{l=j+1}^{m}\mathbb{E}(f_{l}|\mathcal{A})(T_{1}^{p_{1,l}(nQ+i)}\cdots T_{d}^{p_{d,l}(nQ+i)}x).
		\end{split}
		\end{equation*}
		For any $n> L_j$, we define $\mathcal{J}_{j,n}=\mathcal{A}_{g_{j_0}+(p_{1,j}(nQ+i),\cdots,p_{d,j}(nQ+i))}$. 

		Clearly, $$\mathbb{E}(f_{l}|\mathcal{A})(T_{1}^{p_{1,l}(nQ+i)}\cdots T_{d}^{p_{d,l}(nQ+i)}x)$$ is $\mathcal{A}$-measurable for $l>j$.  And $$(f_j-\mathbb{E}(f_j|\mathcal{A}_{g_{j_0}}))(T_{1}^{p_{1,j}(nQ+i)}\cdots T_{d}^{p_{d,j}(nQ+i)}x)$$ is $\mathcal{J}_{j,n-1}(\mathcal{J}_{j,n+1})$-measurable. For $l<j$, there exists $M_j\in \N$ with $M_j>L_j$ such that if $n>M_j$, $f_l(T_{1}^{p_{1,l}(nQ+i)}\cdots T_{d}^{p_{d,l}(nQ+i)}x)$ is $\mathcal{J}_{j,n}$-measurable since when $n>K$, for any $1\le k\le j-1$, $$\deg\big(\sum_{i=1}^{d}A_i(p_{i,k}(n)-p_{i,j}(n))\big)\ge 1$$ and  $$(p_{1,j}(n),\cdots,p_{d,j}(n)) <_{\Phi}(p_{1,k}(n),\cdots,p_{d,k}(n)).$$
		
		To sum up, when $n>M_j$, $X_{j,n}$ is $\mathcal{J}_{n-1}(\mathcal{J}_{n+1})$-measurable. Use that fact that $$(p_{1,j}(nQ+i),\cdots,p_{d,j}(nQ+i))^{-1}\mathcal{A}_{g_{j_0}}=\mathcal{J}_{j,n}$$ and Lemma \ref{lem}, we know
		\begin{equation*}
		\begin{split}
		&\mathbb{E}((p_{1,j}(nQ+i),\cdots,p_{d,j}(nQ+i))\mathbb{E}(f_j|\mathcal{A}_{g_{j_0}})|\mathcal{J}_{j,n})\\&=\mathbb{E}((p_{1,j}(nQ+i),\cdots,p_{d,j}(nQ+i))f_j|\mathcal{J}_{j,n}).
		\end{split}
		\end{equation*}
		 Then $\mathbb{E}(X_{j,n}|\mathcal{J}_{j,n})=0.$ For any $n>m>M_j$, we have $$\mathbb{E}(X_{j,n}X_{j,m})=\mathbb{E}(X_{j,n}\mathbb{E}(X_{j,m}|\mathcal{J}_{j,m}))=0$$
			or
 $$\mathbb{E}(X_{j,n}X_{j,m})=\mathbb{E}(X_{j,m}\mathbb{E}(X_{j,n}|\mathcal{J}_{j,n}))=0.$$
		
		Note the fact that for $\mu$-a.e. $x\in X$, we have 
		\begin{equation*}
			\begin{split}
			&\limsup_{N\rightarrow \infty}\Big|\frac{1}{N}\sum_{n=0}^{N-1}\prod_{k=1}^{j-1}f_k(T_{1}^{p_{1,k}(n)}\cdots T_{d}^{p_{d,k}(n)}x)(f_j-\mathbb{E}(f_j|\mathcal{A}_{g_{j_0}}))(T_{1}^{p_{1,j}(n)}\cdots T_{d}^{p_{d,j}(n)}x)\cdot\\& \ \ \ \ \ \ \ \ \ \ \prod_{l=j+1}^{m}\mathbb{E}(f_{l}|\mathcal{A})(T_{1}^{p_{1,l}(n)}\cdots T_{d}^{p_{d,l}(n)}x)\Big|\\&=\limsup_{N\rightarrow \infty}\Big|\frac{1}{Q}\sum_{i=0}^{Q-1}\frac{1}{N}\sum_{n=0}^{N-1}\prod_{k=1}^{j-1}f_k(T_{1}^{p_{1,k}(nQ+i)}\cdots T_{d}^{p_{d,k}(nQ+i)}x)(f_j-\mathbb{E}(f_j|\mathcal{A}_{g_{j_0}}))\\&\ \ \ \ (T_{1}^{p_{1,j}(nQ+i)}\cdots T_{d}^{p_{d,j}(nQ+i)}x)\cdot \prod_{l=j+1}^{m}\mathbb{E}(f_{l}|\mathcal{A})(T_{1}^{p_{1,l}(nQ+i)}\cdots T_{d}^{p_{d,l}(nQ+i)}x)\Big|.
			\end{split}
		\end{equation*}
		
		By Lemma \ref{lem1}, we can get (\ref{eq3}). This finishes the proof.
	\end{proof}
	
	Now, we are about to prove Theorem \ref{T1}.
	\begin{proof}[Proof of Theorem \ref{T1}]
		Without loss of genarality, we can assume $p_{i,j}(n)\in \Z[n],1\le i\le d,1\le j\le m$ and they satisfy the following properties:
		\begin{itemize}
			\item[(1)]For any $1\le i\le d,l\le j\le m$, $p_{i,j}(0)=0$.
			\item[(2)]For each $1\le j\le m$, $(p_{1,j}(n),\cdots,p_{d,j}(n))$ is not a constant and for any $1\le k,l\le m$ with $k\neq l$, $(p_{1,k}(n)-p_{1,l}(n),\cdots,p_{d,k}(n)-p_{1,l}(n))$ is not a constant.
			\item[(3)]There exist $A_1,\cdots,A_d>0$ such that there exists $N_0\in \N$ such that when $n>N_0$, for each $1\le j\le m$, $$\deg\big(\sum_{i=1}^{d}A_i p_{i,j}(n)\big)\ge 1$$ and there exists $N_1\in \N$\  such that when $n>N_1$, for any $1\le k,l\le m$ with $k\neq l$, $$\deg\big(\sum_{i=1}^{d}A_i(p_{i,k}(n)-p_{i,l}(n))\big)\ge 1.$$
			\item[(4)]Let $\Phi=\{(n_1,\cdots,n_d)\in\Z^d:there\ \  exists\ \  j\in\{0,1,\cdots,d-1\}such\ \  that\ \\\sum_{l=1}^{d-k}A_{l}n_l=0\ for\ k=0,\cdots,j-1\ and\ \sum_{l=1}^{d-j}A_{l}n_l<0\}$. By Example \ref{ex1}, we know it is an algebraic past of $\Z^d$. And $$(p_{1,l}(n),\cdots,p_{d,l}(n)) <_{\Phi}(p_{1,k}(n),\cdots,p_{d,k}(n))$$ for any $1\le k,l\le m$ with $k<l$ when $n>N_2$ where $N_2\in \N$. 
		\end{itemize}
		
		First, we prove the sufficiency.

		Let $f_j=1_{A_j}$ for each $1\le j \le m$ where $A_j\in \B$ and $0<\mu(A_j)$. For any $g\in \Z^d$, $\mathcal{A}_{g}$ is a sub-$\sigma$-algebra generated by $\{hf_l:1\le l\le d,\ h\in \Z^d\ and\ g\le_{\Phi} h\}$. Let $\mathcal{A}=\bigcap_{g\in \Z^d}\mathcal{A}_{g}.$ Note that $\Phi^{-1}$ is still an algebraic past of $\Z^{d}$. Use the definition of $\mathcal{A}$ and Proposition \ref{prop1}, we know $\mathcal{A}\subset P_{\mu}(\Z^d).$
		
	Use Lemma \ref{prop3} and take sum along $j$ for $1\le j \le m$, we get $$\frac{1}{N}\sum_{n=0}^{N-1}\prod_{j=1}^{m}f_j(T_{1}^{p_{1,j}(n)}\cdots T_{d}^{p_{d,j}(n)}x)-\frac{1}{N}\sum_{n=0}^{N-1}\prod_{j=1}^{m}\mathbb{E}(f_j|\mathcal{A})(T_{1}^{p_{1,j}(n)}\cdots T_{d}^{p_{d,j}(n)}x)\rightarrow 0$$ for $\mu$-a.e.$x\in X$ as $N\rightarrow \infty$. Use linear property, sufficient condition and Theorem \ref{thm6}, we know for any $ f_1,\cdots,f_m\in L^{\infty}(X,\B,\mu)$, we have $$\frac{1}{N}\sum_{n=0}^{N-1}\prod_{j=1}^{m}f_j(T_{1}^{p_{1,j}(n)}\cdots T_{d}^{p_{d,j}(n)}x)\rightarrow L(f_1,\cdots,f_m)$$ almost everywhere as $N\rightarrow \infty$ where $L(f_1,\cdots,f_m)\in L^{\infty}(X,\B,\mu)$.
	
	As for necessity, it is clear. This finishes the proof.
	\end{proof}
	\begin{rem}\label{rm1}
		In fact, we can describe the form of $L(f_1,\cdots,f_m)$ when $(X,\B,\mu,\Z^d)$ is a $K$-system.
		
		Next, we verify that $L(f_1,\cdots,f_m)=\prod_{j=1}^{m}\int f_jd\mu$ for any $f_j\in L^{\infty}(X,\B,\mu),1\le j\le m$.  Without loss of genarality, we can assume $f_j\ge 0$ for any $1\le j\le m$.
		
		Select simple function sequences $\{\phi_{j}^{k}\}_{k\ge 1},1\le j\le m$ such that $\phi_{j}^{k}\rightarrow f_j$ in $L^{2}(\mu)$ and $||f_j||_{\infty}\ge||\phi_{j}^{k}||_{\infty}$ for any $k\ge 1$. Let $I_{k}=\int\phi_{1}^{k}d\mu \cdots \int \phi_{m}^{k}d\mu$. Then $$\frac{1}{N}\sum_{n=0}^{N-1}\prod_{j=1}^{m}\phi_{j}^{k}(T_{1}^{p_{1,j}(n)}\cdots T_{d}^{p_{d,j}(n)}x)\rightarrow I_k$$ for $\mu$-a.e.$x\in X$ as $N\rightarrow \infty$. Since $f_1,\cdots,f_m\in L^{\infty}(\mu)$, there exists $K>0$ such that $$K>\max\{\prod_{1\le j\le m,j\neq k}||f||_{\infty}:1\le k \le m\}.$$ By Theorem \ref{thm5}, there exists $C>0$ such that $$||L(f_1,\cdots,f_m)-I_k||_{2}\le CK\sum_{j=1}^{m}||f_j-\phi_{j}^{k}||_{2}.$$ Then $L(f_1,\cdots,f_m)=\prod_{j=1}^{m}\int f_jd\mu$.
	\end{rem}
	
    \section{The proof of Theorem \ref{T2}}
     
    Before the proof, we need a lemma. In the following lemma, we consider such situation:
    
    Under the assumption of Theorem \ref{T2}. Let $f_1,\cdots,f_m\in L^{\infty}(X,\B,\mu)$ and fix them. Assume that $p_{j}(0)=0,1\le j\le m$. To be convenient, we view $\vec{n}$ as $T_{1}^{n_1}\cdots T_{d}^{n_d}$ where $\vec{n}=(n_1,\cdots,n_d)\in \Z^d$. Let $\vec{e_1},\cdots,\vec{e_d}$ denote the natural basis of $\Z^d$. That is, for any $1\le i\le d$, $\vec{e_i}=(n_1,\cdots,n_d)$ where $n_1=\cdots=n_{i-1}=n_{i+1}=\cdots=n_{d}=0$ and $n_i=1$. Moreover, we assume that for each $1\le j\le m$, $p_{j}(n)$ is not a constant and for any $1\le k<l\le m$, if $g(k)=g(l)$, then $p_{k}(n)-p_{l}(n)$ is not a constant. Clearly, there exist $A_1,\cdots,A_d>0$ such that there exists $N_0\in \N$ such that when $n>N_0$, for each $1\le j\le m$, $\deg\big(A_{g(j)}p_{j}(n) \big)\ge 1$ and there exists $N_1\in \N$ such that when $n>N_1$, for any $1\le k,l\le m$ with $k\neq l$, $\deg\big(A_{g(k)}p_{k}(n)-A_{g(l)}p_{l}(n) \big)\ge 1.$
    
    By Example \ref{ex1}, we know $\Phi=\{(n_1,\cdots,n_d)\in\Z^d:there\ exists\ j\in\{0,1,\cdots,d-1\}such\ that\ \sum_{l=1}^{d-k}A_{l}n_l=0\ for\ k=0,\cdots,j-1\ and\ \sum_{l=1}^{d-j}A_{l}n_l<0\}$ is an algebraic past of $\Z^d$. To be convenient, suppose that $p_{l}(n)\vec{e}_{g(l)} <_{\Phi}p_{k}(n)\vec{e}_{g(k)} $ for any $1\le k,l\le m$ with $k<l$ when $n>N_2$ where $N_2\in \N$. For any $g\in \Z^d$, $\mathcal{A}_{g}$ is a sub-$\sigma$-algebra generated by $\{hf_l:1\le l\le d,\ h\in \Z^d\ and\ g\le_{\Phi} h\}$. Let $\mathcal{A}=\bigcap_{g\in \Z^d}\mathcal{A}_{g}.$

    \begin{lemma}\label{prop2}
    	For any $ j\in\{1,\cdots,m\}$, we have
    	\begin{equation*}
    	\begin{split}
    	\frac{1}{N}\sum_{n=0}^{N-1}\prod_{k=1}^{j-1}f_k(T_{g(k)}^{p_k(a_n)}x)(f_j-\mathbb{E}(f_j|\mathcal{A}))(T_{g(j)}^{p_j(a_n)}x)\prod_{l=j+1}^{m}\mathbb{E}(f_{l}|\mathcal{A})(T_{g(l)}^{p_{l}(a_n)}x)\rightarrow 0
    	\end{split}
    	\end{equation*}
    	almost everywhere as $N\rightarrow \infty$.
    \end{lemma}
    \begin{proof}
    	$\forall \epsilon>0$, $\exists g_{j_0}\in \Phi^{-1}$ such that $||\mathbb{E}(f_j|\mathcal{A}_{g_{j_0}})-\mathbb{E}(f_j|\mathcal{A})||_{2}<\epsilon$ by Theorem \ref{thm4}. By Theorem \ref{thm1}, we have
    	\begin{equation*}
    	\begin{split}
    	& \int\Big(\limsup\limits_{N\rightarrow \infty}\Big|\frac{1}{N}\sum_{n=0}^{N-1}\prod_{k=1}^{j-1}T_{g(k)}^{p_k(a_n)}f_k T_{g(j)}^{p_j(a_n)}(\mathbb{E}(f_j|\mathcal{A}_{g_{j_0}})-\mathbb{E}(f_j|\mathcal{A}))\cdot\\&
    	\prod_{l=j+1}^{m}T_{g(l)}^{p_{l}(a_n)}\mathbb{E}(f_{l}|\mathcal{A})\Big|\Big)d\mu
    	\le \prod_{k=1}^{j-1}||f_k||_{\infty} c(2,p_j,T_{g(j)})\epsilon\prod_{l=j+1}^{m}||f_{l}||_{\infty}.
    	\end{split}
    	\end{equation*}
    	
    	So we only need to verify that as $N\rightarrow \infty$,
    	\begin{equation}\label{eq1}
    	\begin{split}
    	\frac{1}{N}\sum_{n=0}^{N-1}\prod_{k=1}^{j-1}f_k(T_{g(k)}^{p_k(a_n)}x) (f_j-\mathbb{E}(f_j|\mathcal{A}_{g_{j_0}}))(T_{g(j)}^{p_j(a_n)}x)\prod_{l=j+1}^{m}\mathbb{E}(f_{l}|\mathcal{A})(T_{g(l)}^{p_{l}(a_n)}x)\rightarrow 0
    	\end{split}
    	\end{equation}
    	almost everywhere. 
    	
    Let $K=\max\{N_0,N_1,N_2\}$. 
     There exists $L_j\in \N$ with $L_j>K$ and $Q\in \N$ such that when $n>L_j$, we have  $g_{j_0}+p_j(a_{(n-1)Q+i})\vec{e}_{g(j)}<_{\Phi}p_j(a_{nQ+i})\vec{e}_{g(j)}$  or $g_{j_0}+p_j(a_{(n+1)Q+i})\vec{e}_{g(j)}<_{\Phi}p_j(a_{nQ+i})\vec{e}_{g(j)}$ for any $0\le i\le Q-1$. Choose any $0\le i\le Q-1$ and fix it. 	Let  $$X_{j,n}=\prod_{k=1}^{j-1}f_k(T_{g(k)}^{p_k(a_{nQ+i})}x)(f_j-\mathbb{E}(f_j|\mathcal{A}_{g_{j_0}}))(T_{g(j)}^{p_j(a_{nQ+i})}x)\prod_{l=j+1}^{m}\mathbb{E}(f_{l}|\mathcal{A})(T_{g(l)}^{p_{l}(a_{nQ+i})}x).$$
    	For any $n>L_j$, we define $\mathcal{J}_{j,n}=\mathcal{A}_{g_{j_0}+p_j(a_{nQ+i})\vec{e}_{g(j)}}$. 
    	
    	Clearly, $\mathbb{E}(f_{l}|\mathcal{A}_j)(T_{g(l)}^{p_{l}(a_{nQ+i})}x)$ is $\mathcal{A}$-measurable for $l>j$. And $(f_j-\mathbb{E}(f_j|\mathcal{A}_{g_{j_0}}))\\(T_{g(j)}^{p_j(a_{nQ+i})}x)$ is $\mathcal{J}_{j,n-1}(\mathcal{J}_{j,n+1})$-measurable. For $l<j$, $T_{g(l)}^{p_l(a_{nQ+i})}f_l$ is $\mathcal{J}_{j,n}$-measurable since there exists $M_j\in \N$ with $M_j>L_j$ such that when $n>M_j$, $p_{j}(n)\vec{e}_{g(j)}+g_{j_0} <_{\Phi}p_{l}(n)\vec{e}_{g(l)} $ for any $1\le l\le j-1$. To sum up, when $n>M_j$, $X_{j,n}$ is $\mathcal{J}_{j,n-1}(\mathcal{J}_{j,n+1})$-measurable. Use that fact that $(p_j(a_{nQ+i})\vec{e}_{g(j)})^{-1}\mathcal{A}_{g_{j_0}}=\mathcal{J}_{j,n}$ and Lemma \ref{lem}, we know $$\mathbb{E}((p_j(a_{nQ+i})\vec{e}_{g(j)})\mathbb{E}(f_j|\mathcal{A}_{g_{j_0}})|\mathcal{J}_{j,n})=\mathbb{E}((p_j(a_{nQ+i})\vec{e}_{g(j)})f_j|\mathcal{J}_{j,n}).$$ So $\mathbb{E}(X_{j,n}|\mathcal{J}_{j,n})=0.$ When $n>m>M_j$, we have $$\mathbb{E}(X_{j,n}X_{j,m})=\mathbb{E}(X_{j,n}\mathbb{E}(X_{j,m}|\mathcal{J}_{j,m}))=0$$ or $$\mathbb{E}(X_{j,n}X_{j,m})=\mathbb{E}(X_{j,m}\mathbb{E}(X_{j,n}|\mathcal{J}_{j,n}))=0.$$ Note the fact that for $\mu$-a.e. $x\in X$, we have 
    	\begin{equation*}
    	\begin{split}
    	&\limsup_{N\rightarrow \infty}\Big|\frac{1}{N}\sum_{n=0}^{N-1}\prod_{k=1}^{j-1}f_k(T_{g(k)}^{p_k(a_n)}x) (f_j-\mathbb{E}(f_j|\mathcal{A}_{g_{j_0}}))(T_{g(j)}^{p_j(a_n)}x)\prod_{l=j+1}^{m}\mathbb{E}(f_{l}|\mathcal{A})(T_{g(l)}^{p_{l}(a_n)}x)\Big|\\&=\limsup_{N\rightarrow \infty}\Big|\frac{1}{Q}\sum_{i=0}^{Q-1}\frac{1}{N}\sum_{n=0}^{N-1}\prod_{k=1}^{j-1}f_k(T_{g(k)}^{p_k(a_{nQ+i})}x)(f_j-\mathbb{E}(f_j|\mathcal{A}_{g_{j_0}}))(T_{g(j)}^{p_j(a_{nQ+i})}x)\\&\ \ \ \  \prod_{l=j+1}^{m}\mathbb{E}(f_{l}|\mathcal{A})(T_{g(l)}^{p_{l}(a_{nQ+i})}x)\Big|.
    	\end{split}
    	\end{equation*}
    	
     By Lemma \ref{lem1}, we can get (\ref{eq1}). This finishes the proof.
    \end{proof}
	Based on the Lemma \ref{prop2} and Theorem \ref{thm3}, the proof of Theorem \ref{T2} is similar to the proof of Theorem \ref{T1}. 
	
	\begin{rem}
		a. The core of proof of Theorem \ref{T2} is the fact that prime sequence is strictly increasing. By this phenomenon, we know that if there exists relative maximal inequality for a strictly increasing(decrease) integer sequence, then one can build a result like Theorem \ref{T2} for this sequence.
		
		b. Under the assumption of Theorem \ref{T2} and $(X,\B,\mu,\Z^d)$ is a $K$-system. Repeat the similar argument of the Remark \ref{rm1}, we know the following result:
		
		If for each $1\le j\le m$, $p_{j}(n)$ is not a constant and for any $1\le k<l\le m$, if $g(k)=g(l)$, then $p_{k}(n)-p_{l}(n)$ is not a constant. Then for any $f_1,\cdots,f_m\in L^{\infty}(X,\B,\mu)$, $$\lim_{N\rightarrow \infty}\frac{1}{N}\sum_{n=0}^{N-1}\prod_{j=1}^{m}f_j(T_{g(j)}^{p_{j}(a_n)}x)=\prod_{j=1}^{m}\int f_jd\mu$$ almost everywhere.
	\end{rem}

   \section*{Acknowledgement}
  The author is supported by NNSF of China (11971455, 12031019, 12090012). The author's thanks go to Professor Wen Huang and Professor Song Shao for their suggestions.

\bibliographystyle{plain}
\bibliography{ref}

\end{document}